\documentclass{amsart}
\usepackage{amssymb}
\usepackage{mathrsfs}
\usepackage{amssymb}
\usepackage{amsfonts}
\usepackage[english]{babel}
\usepackage[T1]{fontenc}
\usepackage[latin1]{inputenc}
\usepackage{fullpage}
\usepackage{amsmath}
\usepackage[all]{xy}
\usepackage{stmaryrd}

\newtheorem{theorem}{Theorem}[section]
\newtheorem{lemma}[theorem]{Lemma}

\theoremstyle{definition}
\newtheorem{definition}[theorem]{Definition}

\newtheorem{proposition}[theorem]{Proposition}

\theoremstyle{remark}
\newtheorem{remark}[theorem]{Remark}

\numberwithin{equation}{section}

\begin{document}

\title{On scalar curvature rigidity of Vacuum Static Spaces}

\author{Jie Qing}
\address{(Jie Qing) Department of Mathematics, University of California, Santa Cruz, CA 95064}
\email{qing@ucsc.edu}

\author{Wei Yuan}
\address{(Wei Yuan) Department of Mathematics, University of California, Santa Cruz, CA 95064}
\email{wyuan2@ucsc.edu}




\keywords{vacuum static space, deformation of metrics, local and global scalar
curvature rigidity}

\thanks{Research of the authors supported by NSF grant DMS-1005295 and DMS-1303543.}

\begin{abstract}
In this paper we extend the local scalar curvature rigidity result in \cite{B-M} to a small domain on general vacuum static spaces, which confirms the interesting 
dichotomy of local surjectivity and local rigidity about the scalar curvature in general in the light of the paper \cite{Corvino}. We obtain the local scalar curvature 
rigidity of bounded domains in hyperbolic spaces. We also obtain the 
global scalar curvature 
rigidity for conformal deformations of metrics in the domains, where the lapse functions are positive, on vacuum static spaces with positive scalar curvature, 
and show such domains are maximal, which generalizes the work in \cite{H-W1}. 
\end{abstract}

\maketitle



\section{Introduction}
The positive mass theorem \cite{S-Y1,S-Y2,Witten} is a fascinating theorem that has been 
pivotal in mathematical relativity. The global scalar curvature rigidity of the Euclidean space $\mathbb{R}^n$ is at the core of the positive mass 
theorem for asymptotically flat manifolds. Analogously the global scalar curvature rigidity of hyperbolic space $\mathbb{H}^n$ is at the core 
of the positive mass theorem for asymptotically hyperbolic manifolds  \cite{Min-Oo1,A-D,Wang-1,C-H,ACG}. This led Min-Oo in \cite{Min-Oo2} to conjecture 
that some global scalar curvature rigidity should also hold for the round hemisphere $\mathbb{S}^n_+$.
\\

In the paper \cite{F-M}, Fischer and Marsden studied the deformations of scalar curvature and 
introduced the notion of static spaces that incidentally is the same notion of vacuum static spaces 
introduced in \cite{H-E,K-O} in mathematical relativity. Fischer and Marsden showed the local 
surjectivity for the scalar curvature as a map from the space of metrics to the space of functions at a non-static metric on a closed
manifold. Corvino in \cite{Corvino} considered compactly supported deformations of metrics and 
extended the local surjectivity result.  
In \cite{F-M}, Fischer and Marsden also observed the local scalar curvature rigidity of closed 
flat manifolds.
Their local scalar curvature rigidity states that, on a closed flat manifold, any metric with nonnegative
scalar curvature that is sufficiently close to a flat metric has to be isometric to the flat metric. This dichotomy of 
local surjectivity and local rigidity about scalar curvature seems extremely intriguing. One wonders if such dichotomy holds in
general based on the work of Corvino \cite{Corvino}.
\\

Min-Oo's conjecture attracted a lot of attentions among geometric analysts. 
It was remarkable that Brendle, Marques and Neves in \cite{B-M-N} 
(see also \cite{C-L-M} for a later developement) discovered that there is even no local scalar curvature rigidity of the round 
hemispheres and constructed counter-examples to Min-Oo's conjecture. Later, in a subsequent paper \cite{B-M}, 
Brendle and Marques established the local scalar curvature rigidity of round spherical caps of some appropriate size (cf. \cite{C-M-T, M-T} for a better
estimate on the size). These developments inspire us to study the local scalar curvature rigidity of general vacuum static spaces.

\begin{theorem}\label{rigidity_static_space}
Suppose that $(M^n,\bar{g}, f)$ ($n\geq 3$) is a vacuum static space. Let $p \in M$ and $f(p) \neq 0$. Then there exist $r_0 > 0$ such that, for each geodesic ball
$B_r(p)$,  there exists $\varepsilon_0>0$ such that, for any metric $g$ on $\overline{B_r(p)}$ satisfying,   
\begin{itemize}
\item $g = \bar{g}$ on $\partial B_r(p)$;
\item $R [g] \geq R [\bar{g}]$ in $B_r(p)$;
\item $H[g] \geq H[\bar g]$ on $\partial B_r(p)$;
\item $||g-\bar{g}||_{C^2(\overline{B_r(p)})} < \varepsilon_0$,
\end{itemize}
there exists a diffeomorphism $\varphi : {B_{r}(p)} \to {B_{r}(p)}$
such that $\varphi|_{\partial B_{r}(p)} = id$ and $\varphi^*g = \bar{g}$ in $B_{r}(p)$,  provided that $r<r_0$.
\end{theorem}

This confirms the dichotomy of local surjectivity and local rigidity about scalar curvature in general in the light of the local surjectivity work of Corvino in \cite{Corvino}. 
In fact Theorem \ref{rigidity_static_space} is stronger than a local rigidity, because it allows the metric $g$ differ from $\bar g$ up to the boundary as long as the mean 
curvature is not less pointwisely on the boundary. Therefore Theorem \ref{rigidity_static_space} is a local rigidity of bounded domains in vacuum static spaces generalizing
the scalar curvature rigidity of bounded domains in particular vacuum static spaces established in \cite{Miao, S-T,B-M}. 
\\

Space forms are the special vacuum static spaces. In Section \ref{space-form} we will discuss the local scalar curvature rigidity of each space form. 
In the Euclidean cases, we are able to obtain the local scalar curvature rigidity, which may be considered as a local version of the rigidity results of  
\cite{Miao,S-T}.  In the hyperbolic cases, it seems that our local scalar curvature rigidity is new and addresses the rigidity problem of the 
positive mass theorem for metrics with corners that is raised in \cite{B-Q}. 
\begin{theorem}\label{Hyperbolic_rigidity}
For $n\geq 3$, let $B_r^{\mathbb{H}}$ be the geodesic ball with radius $r > 0$ in hyperbolic space $\mathbb{H}^n$. 
There exists an $\varepsilon_0 > 0$, such that,  for any metric $g$ on $B_r^{\mathbb{H}}$ satisfying
\begin{itemize}
  \item $g = g_{\mathbb{H}^n}$ on $\partial B_r^{\mathbb{H}}$;
  \item$R[g] \geq -n(n-1)$ in $B_r^{\mathbb{H}}$; 
  \item $H[g] \geq H[g_{\mathbb{H}^n}]$ on $\partial B_r^{\mathbb{H}}$;
  \item $|| g - g_{\mathbb{H}^n}||_{C^2 (\overline{B_r^{\mathbb{H}}})} < \varepsilon_0$, 
\end{itemize}
there is a diffeomorphism $\varphi: {B_r^{\mathbb{H}}} \to {B_r^{\mathbb{H}}}$ such that $\varphi^*g = g_{\mathbb{H}^n}$ in $B_r^{\mathbb{H}}$
and $\varphi = id$ on $\partial B_r^{\mathbb{H}}$.
\end{theorem}
In the spherical cases, the local scalar curvature rigidity is established in \cite{B-M,C-M-T}. 
It remains intriguing to find out whether one can identity the size of the spherical cap on which the local scalar curvature 
rigidity first fails to be valid. More interestingly it remains open that whether there is global scalar curvature rigidity for bounded domains in the hyperbolic cases
and bounded domains of appropriate size in the spherical cases. 
\\

It has been noticed that it is an interesting intermediate step to consider scalar curvature rigidity among conformal deformations. Hang and Wang in \cite{H-W1},
for instance, obtained the global scalar curvature rigidity among conformal metrics for the round hemisphere (a weaker version of Min-Oo's conjecture). 
They in fact also showed that the local scalar curvature rigidity even among conformal metrics is no longer true for the round metric on any spherical cap 
bigger than the hemisphere. We observe that the rigidity for conformal deformations in \cite{H-W1}
can be extended for general vacuum static spaces with positive scalar curvature.

\begin{theorem}\label{conformal rigidity_static_space}
Let $(M^n,\bar{g}, f)$ be a complete $n$-dimensional vacuum static space with $R[\bar{g}] > 0$ ($n \geq 2$). Assume the level set
$\Omega^+ = \{ x \in M^n : f(x) > 0 \}$ is a pre-compact subset in $M$. Then 
if a metric $g \in [\bar{g}]$ on $M$ satisfies that
\begin{itemize}
\item $R[g] \geq R[\bar{g}]$ in $\Omega^+$,
\item $g$ and $\bar{g}$ induced the same metric on $\partial \Omega^+$, and
\item $H[g] = H[\bar{g}]$ on $\partial \Omega^+$,
\end{itemize}
then $g=\bar{g}$. On the other hand, for any open domain $\Omega$ in $M^n$ that contains $\overline{\Omega^+}$, 
there is a smooth metric $g \in [\bar{g}]$ such that 
\begin{itemize}
\item $R[g]  \geq R[\bar{g}]$ in $\Omega$, $R[g] > R[\bar g]$ at some point in $\Omega$, and
\item $\textup{supp}(g - \bar g)\subset \Omega$.
\end{itemize}
\end{theorem}

Our proof of the rigidity in Theorem \ref{conformal rigidity_static_space} only uses the maximum principles, 
which seems to be more straightforward than the proof used in \cite{H-W1} in the case of the round hemisphere\footnote{After we posted the previous version of this
paper on arXiv, we were informed that the rigidity part of Theorem \ref{conformal rigidity_static_space} had been known in \cite{B-M-V}.}. The construction in the second part of
Theorem \ref{conformal rigidity_static_space} is based on the idea in \cite{B-M-N}. For a detailed history of the study of the scalar curvature rigidity phenomena 
and the solution of Min-Oo's conjecture, readers are referred to the excellent survey article \cite{Brendle} of Brendle.
\\

\paragraph{\textbf{Acknowledgement}}

We would like to thank Professors Justin Corvino, Pengzi Miao and Carla Cerderbaum for their interests and stimulating discussions. 
\\


\section{Definition of vacuum static spaces}

Static spaces are time slices in static space-times, which are the important global solutions to Einstein equations 
in general relativity. A space-time metric
$$
\hat g = - f^2 dt^2 + \bar g
$$ 
on $\mathbb{R}\times M^n$ is said to be vacuum static if the so-called lapse 
function $f$ and the Riemannian metric $\bar g$ are independent of the time $t$ and the space-time metric $\hat g$
satisfies the vacuum Einstein equation
\begin{equation}\label{equ:einstein}
\textup{Ric}[\hat g] - \frac{1}{2}R[\hat g] \hat{g} + \Lambda \hat g = 0.
\end{equation} 
On the time slice $(M^n, \ \bar g)$ of a vacuum static space-time the Einstein equation \eqref{equ:einstein} turns into the vacuum static equation:
\begin{equation}\label{static_2}
\nabla^2 f -  \left(\textup{Ric}[\bar g] - \frac {R[\bar g]}{n-1} \bar g \right)f = 0\quad \text{on $M^n$}.
\end{equation}

\begin{definition} (\cite{H-E,K-O}) We say $(M^n, \ \bar g, \ f)$ is a \emph{vacuum static space}, if $(M^n, \bar g)$ is a complete Riemannian manifold and the 
smooth function $f \not\equiv 0$ satifies the equation \eqref{static_2}. The function is called the lapse function for the vacuum static space $(M^n, \ \bar g)$. 
\end{definition}

In \cite{F-M}, Fischer and Marsden considered the scalar curvature
$$R[g] : \mathfrak M \rightarrow C^{\infty} (M)
$$
as a map from the space $\mathfrak M$ of Riemannian metrics and calculated its derivative
$$
\gamma_{g} = D R[g] : S_2(M) \rightarrow C^{\infty}(M)
$$
and the $L^2$-formal adjoint  
$$
\gamma^*_{g}f =   \nabla^2 f - g\Delta f - f\cdot \textup{Ric}[g] : C^{\infty} (M) \rightarrow S_2(M).
$$
It is very intriguing that a complete Riemannian manifold $(M^n, \ \bar g, \ f)$ is vacuum static if and only if $f$ is in the kernel of the operator $\gamma_{\bar g}^*$.
\\

Space forms are examples of vacuum static spaces. Locally conformally flat vacuum static spaces have been completely classified in \cite{Kobayashi, Lafontaine} in 1980s. 
Recently, Bach flat vacuum static spaces have also been classified in \cite{Q-Y} based on the idea used in \cite{C-C1, C-C2}. It is clear that 
there are many non-compact vacuum static spaces which are not necessarily locally conformally flat (cf.  \cite{Q-Y}).
\\


\section{Local scalar curvature rigidity in general}

In this section we will investigate the local scalar curvature rigidity phenomenon for general vacuum static spaces.  For convenience of readers, we will
present the calculations in \cite{F-M} and \cite{B-M-N, B-M} (cf. see also \cite{C-M-T, M-T}) for general vacuum static spaces. We first recall the deformations of scalar curvature.
In this paper we use the conventions that Greek indices run through $1,2, \cdots, n$ while Latin indices run through $1,2,\cdots,n-1$.

\begin{lemma}\label{scalar_variation_formulae} (\cite[Lemma 2.2 and Lemma 7.2]{F-M}) For the deformation of metrics $g_t = g + t h$, we have
\begin{align}\label{1-scalar-cur-deform}
\left.\frac{d}{dt}\right|_{t=0}R[g_t] = DR[g] (h) = - \Delta (\textup{Tr} h) + \delta\delta h - \textup{Ric}\cdot h
\end{align}
and
\begin{align}
\left.\frac{d^2}{dt^2}\right|_{t=0} R[g_t] = & D^2 R[g](h,h) = - 2DR[g] (h^2) - \Delta (|h|^2) - \frac{1}{2} |\nabla h|^2 - \frac{1}{2}|d(\textup{Tr} h)|^2\\ \notag & + 2 h\cdot
\nabla^2(\textup{Tr} h)  - 2 \delta h\cdot d(\textup{Tr} h) + \nabla_\alpha h_{\beta\gamma} \, \nabla^{\beta} h^{\alpha\gamma},
\end{align}
where $(\delta h )_\alpha = - \nabla^\beta h_{\alpha\beta}$.
\end{lemma}

Note that the operator $\Delta$ differs from that in \cite{F-M} by a sign. 
Let $(M^n, \ \bar{g},\ f)$ ($n \geq 3$) be a vacuum static space and $\Omega$ ia subdomain in $M^n$. As in \cite{F-M,B-M-N,B-M, M-T}, we consider
the functional
\begin{align}\label{the functional}
\mathscr{F}[g] = \int_{\Omega} fR[g]d\textup{vol} [\bar{g}],
\end{align}
where $d\textup{vol}[\bar{g}]$ is the volume element with respect to the static metric $\bar{g}$ instead of $g$. 
\\

It is well known that such geometric problems need to appropriately fix a gauge in order to derive rigidity results. In \cite{F-M} they relied on the slice theorem in \cite{Ebin} for closed
manifolds. The following lemma from \cite{B-M} is a version of the slice theorem that is applicable to domains instead of closed manifolds without boundary . 

\begin{lemma} \label{slice} (\cite[Proposition 11]{B-M}) Suppose that $\Omega$ is a domain in a Riemannian manifold $(M^n, \ \bar g)$. 
Fix a real number $p > n$, there exists an $\varepsilon > 0$, such that for a metric $g$ on $\Omega$ with $$||g - \bar{g}||_{W^{2,p}({\Omega}, \bar{g})} < \varepsilon,$$ there exists a diffeomorphism $\varphi: {\Omega} \rightarrow {\Omega}$ such that $\varphi|_{_{\partial\Omega}} = id$ and $h = \varphi^*(g) - \bar{g}$ is divergence-free in $\Omega$ with 
respect to $\bar{g}$. Moreover, $$||h||_{W^{2,p}({\Omega}, \bar{g})} \leq C ||g - \bar{g}||_{W^{2,p}({\Omega}, \bar{g})},$$ for some constant $C > 0$ that only depends on ${\Omega}$.
\end{lemma}

For the convenience of calculations we are using a Fermi coordinate of the boundary $\partial\Omega$ with respect to the vacuum static metric $\bar g$ such that
$\partial_n = \partial_\nu$ on the boundary $\partial\Omega$.
Let $A$ be the second fundamental form and $H$ be mean curvatures of $\partial\Omega$. In the following calculations from now on in this section everything is 
with respect to the vacuum static metric $\bar g$ unless it will be indicated otherwise. But, first, in the light of Lemma \ref{slice}, 
we may assume that 
\begin{equation}\label{t-boundary-condition}
\delta h = 0 \text{ in $\Omega$ and } h_{ij} = 0 \text{ on $\partial\Omega$}.
\end{equation}
We would like to mention that it is not necessarily true that $h$ vanishes on the boundary after requiring $\delta h = 0$ in $\Omega$. For the convenience of readers, we 
present the calculations:
\begin{equation}\label{covariant-derivative-at-bdy}
\aligned
h_{ij, k} & = A_{jk} h_{in} + A_{ik}h_{jn}\\
h_{in,}^{\quad i} & = (\nabla^{\partial\Omega})^i h_{in} + H h_{nn}\\
h_{nn,i} & = \partial_i h_{nn} - 2 A_{ij}h_n^{\ j}\\
h_{j \  , i} ^{\ j} & = 2 A_{ij}h_{n}^{\ j}\\
h_{\alpha n,}^{\quad n} &= - h_{\alpha i,}^{\quad i}.
\endaligned 
\end{equation}

\begin{lemma} \label{variations} Suppose that $(M^n, \ \bar g, \ f)$ is a vacuum static space and that $g_t = \bar g + t h$ is a deformation. Then
\begin{align}\label{1-variation}
\left.\frac{d}{dt}\right|_{t=0} \mathscr{F}[g_t] (h)= \int_{\partial\Omega} \left( (\textup{Tr} h) \partial_{{\nu}}f - f \partial_{{\nu}} (\textup{Tr} h) - 
h({\nu}, \nabla f) - f\delta h\cdot {\nu} \right)d\sigma[\bar{g}]
\end{align}
and, if in addition one assumes $\delta h = 0$, 
\begin{equation}\label{2-variation} 
\aligned
\left.\frac{d^2}{dt^2}\right|_{t=0} \mathscr{F}[g_t] (h, h) & =  -\frac{1}{2} \int_{\Omega} (|\nabla h|^2 + |d(\textup{Tr} h)|^2 - 2 \mathscr{R} (h,h))f d\textup{vol}[\bar{g}] \\ 
 + \int_{\partial\Omega} (f ( \partial_{\nu}(|h|^2) + \delta(h^2)\cdot \nu & + 2h(\nabla(\textup{Tr} h), \nu))
+ h^2(\nu, \nabla f)- |h|^2 \partial_{\nu} f - 2 (\textup{Tr} h) h(\nu, \nabla f))d\sigma[\bar{g}],
\endaligned
\end{equation}
where $\mathscr{R}(h,h) = R_{\alpha\beta\gamma\delta}h^{\alpha\gamma}h^{\beta\delta} + 2 (\textup{Tr} h) \textup{Ric} \cdot h - \frac{2R}{n-1} (\textup{Tr} h)^2$.
\end{lemma}
\begin{proof}
One only needs to apply Lemma \ref{scalar_variation_formulae} and perform integrating by parts. One calculation is worth to present here.
\begin{align*}
&\int_{\Omega} \left( \nabla_{\alpha} h_{\beta\gamma} \cdot \nabla^{\beta} h^{\alpha\gamma}\right) f d\textup{vol}_{\bar{g}} \\
=& \int_{\Omega} \left( (\nabla^2 f- f \textup{Ric}[\bar{g}])\cdot h^2  + f R_{\alpha\beta\gamma\delta}h^{\alpha\gamma}h^{\beta\delta}\right) 
d\textup{vol}_{\bar{g}} - \int_{\partial\Omega} 
\left( h^2 (\nu, \nabla f) + f \delta (h^2)\cdot\nu \right)d\sigma_{\bar{g}}\\
=& \int_{\Omega} \left( |h|^2 \Delta f + f R_{\alpha\beta\gamma\delta}h^{\alpha\gamma}h^{\beta\delta}\right) d\textup{vol}_{\bar{g}} - \int_{\partial\Omega} 
\left( h^2(\nu, \nabla f) + f \delta (h^2)\cdot\nu \right)d\sigma[\bar{g}],
\end{align*}
where the static equation \eqref{static_2} and the Ricci identity in Riemannian geometry are used. One may also use \eqref{1-variation} to 
handle the term $\int_\Omega DR[\bar g] (h^2)d\textup{vol}[\bar g]$.
\end{proof}

The following expansion of the mean curvature from \cite{B-M} gives us the first and second deformation of the mean curvature.

\begin{lemma} (\cite[Proposition 5]{B-M}) Suppose that $g = \bar g + h$ be another metric on a domain $\Omega$ in a Riemannian manifold $(M^n, \ \bar g)$. 
Assume that $h|_{T\partial\Omega} = 0$. Then
\begin{equation}\label{expansion-mc}
\aligned
H[g] - H[\bar g] & = (\frac 12 H[\bar g]h_{nn} - h_{in,}^{\quad i} + \frac 12h_{i \ , n}^{\ i}) \\ 
+ \frac 12 ((-\frac 14h_{nn}^2 + h_{in}h^i_{\ n})H[\bar g] & + h_{nn}(h_{in,}^{\quad i} - \frac 12h_{i \ , n}^{\ i}))
+ O(|h|^2(|\nabla h| + |h|)),
\endaligned
\end{equation}
where 
$$
|O(|h|^2(|\nabla h| + |h|))|\leq C (|h|^2(|\nabla h| + |h|))
$$
for some constant $C$ that only depends on $n$.
\end{lemma}

The vacuum metric $\bar g$ is not a critical point for the functional $\mathscr{F}$ according to \eqref{1-variation}, instead it follows from \eqref{expansion-mc}
that  
\begin{equation}\label{some-thing}
\left.\frac d{dt}\right|_{t=0}(\int_\Omega fR[g_t]d\textup{vol}[\bar g] + 2\int_{\partial\Omega}fH[g]d\sigma[\bar g]) = 0
\end{equation} 
for $g_t = \bar g + th$, where $h|_{T\partial\Omega}$ vanishes, as observed in \cite{C-M-T}. An immediate consequence of \eqref{expansion-mc} is 
\begin{equation}\label{main-use-mc}
\aligned
(2 - h_{nn})(H[g] - H[\bar g])  & = - (1 - h_{nn})(2h_{ni,}^{\quad i} - h_{i \ , n}^{\ i}) \\
+ (h_{nn} - \frac 34h_{nn}^2 + h_{in} & h^i_{\ n}) H[\bar g] + O(|h|^2(|\nabla h|+ |h|)).
\endaligned
\end{equation}
For the convenience we denote 
\begin{align}\label{I_Omega}
I_{\Omega} = \frac 14 \int_\Omega 
\left( |\nabla h|^2 + |d(\textup{Tr} h)|^2 -2\mathscr{R}(h,h) \right) f d\textup{vol}[\bar{g}]
\end{align}
and
\begin{align}
B_{\Omega}=& \int_{\partial\Omega} \left( - (\textup{Tr} h) \partial_{\nu} f + h(\nu, \nabla f)- \frac{1}{2} h^2(\nu, \nabla f) 
+ \frac{1}{2} |h|^2 \partial_{\nu} f  + (\textup{Tr} h) h(\nu, \nabla f) \right) d\sigma[\bar{g}]\\ \notag
&+ \int_{\partial\Omega} \left( \partial_{\nu} (\textup{Tr} h) - \frac{1}{2} \partial_{\nu}(|h|^2)  - \frac{1}{2} \delta(h^2)\cdot\nu 
- h(\nabla(\textup{Tr} h), \nu) \right)f  d\sigma[\bar{g}].
\end{align}
Therefore we may write
\begin{equation}\label{taylor-2}
\mathscr{F}[g] - \mathscr{F}[\bar{g}] - \mathscr{F}'[\bar{g}] (h) - \frac{1}{2} \mathscr{F}''[\bar{g}](h,h) = \int_{\Omega} \left( R[g] - R[\bar{g}] \right)fd\textup{vol}[\bar g]  
+ I_\Omega + B_\Omega
\end{equation}
when $h = g - \bar g$ satisfies \eqref{t-boundary-condition}.
\\

\begin{proposition}\label{I_Sigma}
Assume that 
$$
|h| < \frac{1}{2}, \quad \delta h = 0, \text{ and } h_{T\partial\Omega} = 0.
$$
Then
\begin{align}\label{bdy-term}
B_\Omega =& \int_{\partial\Omega} \left( (2 - h_{nn}) \left( H[g] - H[\bar g] \right)  + \frac{1}{2} A^{ij} h_{in} h_{jn} +  \frac{1}{4}|h|^2 H[\bar g]\right)f d\sigma[\bar{g}] \\
+ \int_{\partial\Omega} & \left( h_{nn}h(\nu, \nabla f) + \frac{1}{4}\left(|h|^2 - h^2_{nn}  \right) \partial_{\nu} f \right) d\sigma[\bar{g}] 
+ O \left( \int_{\partial\Omega} |h|^2 (|\nabla h| +  |h|) d\sigma[\bar{g}] \right), \notag
\end{align}
where 
$$
|O(\int_\Omega |h|^2(|\nabla h| + |h|)d\textup{vol}[\bar g]) |\leq C \int_\Omega |h|^2(|\nabla h| + |h|)d\textup{vol}[\bar g]
$$
for some constant $C$ that only depends on $n$.
\end{proposition}

\begin{proof} First we calculate
\begin{align*}
B^1_\Omega :&= \int_{\partial\Omega}  \left( - (\textup{Tr} h) \partial_{\nu} f + h(\nu, \nabla f) + \frac{1}{2} |h|^2 \partial_{\nu} f - \frac{1}{2}h^2(\nu, \nabla f) 
+ (\textup{Tr} h)h(\nu, \nabla f) \right) d\sigma[\bar{g}]\\
&= \int_{\partial\Omega} \left( \left( h^2_{nn} + \frac{1}{2}h_{ni}h^i_{\ n} \right)\partial_\nu f + 
(1 + \frac{1}{2} h_{nn} ) h_{in} \cdot \partial^i f\right) d\sigma[\bar{g}].
\end{align*}
To get the second part of $B_\Omega$ we calculate
$$
\aligned
\nu(\textup{Tr}h) &= h_{n \ , n}^{\ n} + h_{i \ , n}^{\ i}\\
- \frac 12\nu(|h|^2) & = -h_{nn}h_{n \ , n}^{\ n} - 2h_n^{\ i}h_{ni,n}\\
-\frac 12\delta(h^2)\cdot\nu & = \frac 12h_{nn}h_{n \ , n}^{\ n} +\frac 12h_n^{\ i}h_{nn, i} +\frac 12 h_n^{\ i}h_{ni,n}\\
-h(\nabla(\textup{Tr} h), \nu) & = - h_{nn}(h_{nn,n} + h_{i \ , n}^{\ i}) - h_n^{\ i}(h_{nn, i} + h_{j \ , i}^{\ j}).
\endaligned
$$ 
Then, using \eqref{covariant-derivative-at-bdy}, we obtain
\begin{align*}
B_\Omega^2 :=& \int_{\partial\Omega}  \left( \partial_{\nu} (\textup{Tr} h) - \frac{1}{2} \partial_{\nu}(|h|^2)  - \frac{1}{2} \delta(h^2)\cdot\nu 
- h(\nabla(\textup{Tr} h), \nu) \right)f d\sigma[\bar{g}]\\
=& \int_{\partial\Omega} \left((\nabla^{\partial\Omega})^i \left( (1 - \frac{1}{2}h_{nn}) h_{in} \right) - ( 1 - h_{nn} )
\left( 2 h_{n \ , i}^{\ i} - h_{i \ , n}^{\ i} \right) \right) f d\sigma[\bar{g}] \\
&\quad\quad\quad\quad + \int_{\partial\Omega} \left( (1 - \frac{1}{2}h_{nn})h_{nn} H+ \frac{1}{2} A^{ij} h_{in} h_{jn} 
+ \frac{3}{2} Hh_{ni}h^i_{\ n} \right) f d\sigma[\bar{g}]\\
= & \int_{\partial\Omega} \left(- \left(1 - \frac{1}{2}h_{nn}\right) h_{in}\partial^i f  - ( 1 - h_{nn} )\left( 2 h_{n \ , i}^{\ i} - h_{i \ , n}^{\ i} \right) f \right) d\sigma[\bar{g}] \\
&\quad\quad\quad\quad + \int_{\partial\Omega} \left( (1 - \frac{1}{2}h_{nn})h_{nn} H+ \frac{1}{2} A^{ij} h_{in} h_{jn} 
+ \frac{3}{2} Hh_{ni}h^i_{\ n} \right) f d\sigma[\bar{g}].
\end{align*}
Therefore, adding $B_\Omega^1$ and $B_\Omega^2$, we arrive at
\begin{align*}
B_\Omega =& \int_{\partial\Omega}  \left( (1 - \frac{1}{2} h_{nn})h_{nn} H+ \frac{1}{2} A^{ij} h_{in} h_{jn}
+ \frac{3}{2} Hh_{ni}h^i_{\ n} \right) f d\sigma[\bar{g}] \\
& \quad\quad\quad\quad-  \int_{\partial\Omega} ( 1 - h_{nn} )\left( 2 h_{n \ , i}^{\ i} - h_{i \ , n}^{\ i} \right) f  d\sigma[\bar{g}] \\
& +  \int_{\partial\Omega} \left( h_{nn}h_{in} \partial^i f + \left(h^2_{nn} + \frac{1}{2}h_{ni}h^i_{\ n} \right) \partial_\nu f \right) d\sigma[\bar{g}].
\end{align*}
Finally, using \eqref{main-use-mc}, we may finish the calculation and establish \eqref{bdy-term}.
\end{proof}

Now we are ready to show the boundary integral is non-negative for 
small geodesic balls in a vacuum static spaces.

\begin{proposition}\label{boundary-est}
Suppose $(M^n, \ \bar g, \ f)$ is  a vacuum static space and that $g= \bar g+ h$ where
$$
|h| + |\nabla h| < \frac{1}{2}, \quad \delta h = 0, \text{ and } h_{T\partial\Omega} = 0.
$$
Then, for a $p_0\in M^n$ where $f(p_0)>0$,  there exists $r_ 0>0$ and $C>0$ such that
\begin{align*}
B_{B_r(p_0)} \geq - C \int_{\partial B_r(p_0)} |h|^2 (|\nabla h| +  |h|) d\sigma[\bar{g}] 
\end{align*}
for any geodesic ball $B_r(p_0)$ with $r<r_0$, 
provided that 
$$
H[g] \geq H[\bar{g}] \text{ and on $\partial B_r(p_0)$}.
$$
\end{proposition}

\begin{proof} First of all, for $p_0\in M^n$ with $f(p_0) >0$, one may have $r_1$ such that
$$
f(p) \geq \epsilon_1 \text{ and } f(p) + |\nabla f|(p) \leq \beta_1
$$
for $p\in B_{r_1}(p_0)$ and positive constants $\epsilon_1$ and $\beta_1$. Hence, from \eqref{bdy-term}, we obtain
$$
\aligned
B_{B_r(p_0)} & \geq  \int_{\partial B_r(p_0)} \left( \frac{1}{2} A^{ij} h_{in} h_{jn} +  \frac{1}{4}|h|^2 H[\bar g]\right)f d\sigma[\bar{g}] 
- C\int_{\partial B_r(p_0)}|h|^2d\textup{vol}[\bar g]  \\
& \quad\quad - C \int_{\partial B_r(p_0)} |h|^2 (|\nabla h| +  |h|) d\sigma[\bar{g}] 
\endaligned
$$
for some constant $C>0$, here we have used the assumptions that $(2 - h_{nn})(H[g] - H[\bar g])\geq 0$.
Therefore it is easy from here to finish the proof based on the geometry of small geodesic balls. Namely,
\begin{align*}
A_{ij} = \frac{1}{r} \bar{g}_{ij} + O(r) \text{ and }
H[\bar{g}] = \frac{n-1}{r} + O (r).
\end{align*}
\end{proof}

Before we give the estimate on the interior term, we study the following eigenvalue problem of the Laplace 
on symmetric 2-tensors. Namely we consider
\begin{align} \label{L-laplace}
\mu (\Omega) = \inf\{ \frac{\int_{\Omega} \frac 12|\nabla h|^2 d\textup{vol}[\bar{g}]}{\int_{\Omega} |h|^2 d\text{vol}[\bar{g}]}: 
h \not\equiv 0 \text{ and } h_{T\partial\Omega} = 0\}
\end{align}
The following is an easy but very useful fact to us (please see a similar result in \cite{K-P}).

\begin{lemma}\label{Euclidean-eigen} Suppose $(M^n, \ \bar g, \ f)$ is  a vacuum static space 
and that $B_r(p) $ is a geodesic ball of radius $r$ in $M^n$. Then, there are constants $r_0$ and $c_0$ 
such that
\begin{equation}\label{lambda-ball}
\mu (B_r(p)) \geq \frac {c_0}{r^2}
\end{equation}
for all point $p\in M^n$ and $r<r_0$.
\end{lemma}

\begin{proof} We first observe that there are constant $r_0$ and $c_1$ such that 
$$
\mu(B_r(p)) \geq c_1\mu(B^0_r(0))
$$
for all $p\in M^n$ and $r<r_0$, where $\mu(B^0_r(0))$ is the first eigenvalue for the Euclidean ball $B^0_r(0)$
with respect to the Euclidean metric.  Therefore it suffices to show \eqref{lambda-ball} for Euclidean balls with respect to the 
Euclidean metric. In fact, by scaling property, we simply need to show 
$$
\mu(B^0_1) >0.
$$
For this purpose, we consider the functional
\begin{align*}
J(h) = \frac{\int_{B_1^0} \frac 12 |\nabla h|^2 dx}{\int_{B_1^0} |h|^2 dx}, \ \ \ \forall h \in \mathscr{W}
\end{align*}
where $$\mathscr{W} = \{ h \in W^{1,2}(B_1^0): h \not\equiv 0, h|_{T{\partial B_1^0}} = 0 \}.$$
Then the Euler-Lagrange equation for the minimizers of  $J$ is
\begin{align*}
\begin{cases}
\Delta h + \mu(B_1^0) h = 0 \ \text{ in $B_1^0$} \\
h_{ij} = 0 \text{ and } \partial_\nu h_{in} = 0 \ \text{ on $\partial B_1^0$},
\end{cases}
\end{align*}
where $\mu (B_1^0) = \inf_{h \in \mathscr{W}} J(h) \geq 0$, if we use the spherical coordinate $\{\theta^1, \cdots, \theta^{n-1}, \nu\}$ on 
the unit sphere. What we need to show is that $\mu(B_1^0)$ is in fact positive. Assume otherwise
$\mu(B_1^0) = 0$. Then we may easily see that, by integral by parts, the eigen-tensor $h$ has to be parallel (constant) in $B_1^0$, which forces 
$h\equiv 0$ since all the tangent vectors at the boundary $\partial B_1^0$ together span the full space $\mathbb{R}^n$ (One may take the 
advantage to ignore the base point for a vector in the Euclidean space) . This finishes the proof.
\end{proof}

We remark that, in case the domain $\Omega$ is a square in the plane $\mathbb{R}^2$, the first eigenvalue $\mu$ is zero. Consequently, we have

\begin{proposition}\label{interior-est}  Suppose $(M^n, \ \bar g, \ f)$ is  a vacuum static space and that $g= \bar g+ h$ where
$$
\delta h = 0, \text{ and } h_{T\partial\Omega} = 0.
$$
Then, for a $p_0\in M^n$ where $f(p_0)>0$,  there exists $r_ 0>0$ such that
\begin{align*}
I_{B_r(p_0)} \geq \frac{1}{8}\int_{B_r(p_0)}(|\nabla h|^2 + |h|^2)d\textup{vol}[\bar g]
\end{align*}
for a geodesic ball $B_r(p_0)$ with radius $r< r_0$. 
\end{proposition}

\begin{proof}
Recall \eqref{I_Omega},
\begin{align*}
I_{B_r(p_0)} = \frac{1}{4}  \int_{B_r(p_0)} \left( |\nabla h|^2 + |d(tr h)|^2 -2\mathscr{R}(h,h) \right) f d\textup{vol}[\bar{g}].
\end{align*}
Clearly there is a constant $C$ ($C$ depends on $(M^n, \ \bar g)$) such that
\begin{align*}
2\mathscr{R}(h,h) \leq C \ |h|^2.
\end{align*}
Hence
\begin{align*}
I_{B_r(p_0)}  &\geq \frac{1}{4}  \int_{B_r(p_0)} \left( |\nabla h|^2 - C  |h|^2 \right) f dvol_{\bar{g}}\\
&= \frac{1}{8} \int_{B_r(p_0)} \left( |\nabla h|^2 + |h|^2 \right) f dvol_{\bar{g}} + \frac{1}{8} \int_{B_r(p_0)} \left( |\nabla h|^2 - (2C  + 1) |h|^2 \right) f dvol_{\bar{g}}.
\end{align*}
The rest of proof easily follows from Lemma \ref{Euclidean-eigen}. 
\end{proof}

Now we are ready to prove the main theorem following the approach from \cite{B-M-N, B-M, C-M-T} .

\begin{proof}[Proof of Theorem \ref{rigidity_static_space}]

First, due to the assumption that $||g - \bar{g}||_{C^2(\overline{B_r(p_0)})}$ is sufficiently small, in the light of Lemma \ref{slice}, we may assume $g = \bar g + h$, 
where $h$ satisfies
$$
\delta h = 0 \text{ in $B_r(p_0)$ and } h|_{T\partial B_r(p_0)} = 0.
$$
Following the approach in \cite{B-M-N, B-M, C-M-T}, we have
\begin{align*}
\mathscr{F}(g) - \mathscr{F}(\bar{g}) - \mathscr{F}'(\bar{g})\cdot h - \frac{1}{2} \mathscr{F}''(\bar{g})\cdot (h, h) \leq C \|h\|_{C^2(\overline{B_r(p_0)})}\int_{B_r(p_0)}
(|\nabla h|^2 + |h|^2)d\text{vol}[\bar g]. 
\end{align*}
On the other hand, by Proposition \ref{boundary-est} and Proposition \ref{interior-est}, one arrives at
\begin{align*}
\int_{B_r(p_0)}
(|\nabla h|^2 + |h|^2)d\text{vol}[\bar g] & \leq C ||h||_{C^2(\overline{B_r(p_0)})}(\int_{B_r(p_0)}
(|\nabla h|^2 + |h|^2)d\text{vol}[\bar g] + \int_{\partial B_r(p_0)}|h|^2d\sigma[\bar g]) \\
& \leq  C_1 ||h||_{C^2(\overline{B_r(p_0)})} \int_{B_r(p_0)}
(|\nabla h|^2 + |h|^2)d\text{vol}[\bar g],
\end{align*}
by Trace Theorem of Sobolev spaces, which implies that $h \equiv 0$, when $\|h\|_{C^2(\overline{B_r(p_0)})}$ is small enough. Thus the proof is complete.
\end{proof}


\section{Local scalar curvature rigidity of space forms}\label{space-form}

In the previous section, we investigated the local scalar curvature rigidity of domains of sufficiently small size in general vacuum static spaces. In this section we consider 
the local scalar curvature rigidity of space forms. 

\subsection{Euclidean spaces} In this subsection, we consider local scalar curvature rigidity of domains in the Euclidean space $\mathbb{R}^n$. In Euclidean cases it 
turns out one has the local scalar curvature rigidity of bounded domain of any size, which may be compared with the rigidity result of closed flat spaces in \cite{F-M} 
but a much a weaker version of the rigidity in \cite{Miao} (cf. also \cite{S-T}), where the positive mass theorem is employed. In Euclidean cases the lapse function $f$ may be taken 
to be $1$. We calculate from \eqref{I_Omega} and \eqref{bdy-term} that
$$
I_{\Omega}  =  \frac 14 \int_\Omega \left( |\nabla h|^2 + |d(\textup{Tr} h)|^2 \right) dx
$$
and
$$
\aligned
B_\Omega = &  \int_{\partial\Omega} \left( (2 - h_{nn}) \left( H[g] - H[g_{\mathbb{R}^n}] \right)  + \frac{1}{2} A^{ij} h_{in} h_{jn} 
+  \frac{1}{4}|h|^2 H[g_{\mathbb{R}^n}]\right) d\theta \\
& \quad\quad + O \left( \int_{\partial\Omega} |h|^2 (|\nabla h| +  |h|) d\theta\right).
\endaligned
$$

\begin{theorem}\label{euclidean} Let $\Omega$ be a bounded smooth domain in the Euclidean space $\mathbb{R}^n$. Assume that
\begin{equation}\label{weak-convex}
A + \frac 12Hg_{\mathbb{R}^n} \geq 0 \text{ on $\partial\Omega$}.
\end{equation}
Then there is $\epsilon>0$ such that, 
for any Riemnnian metric $g$ on $\overline{\Omega}$ satisfying
\begin{itemize}
\item $g = g_{\mathbb{R}^n}$ on $\partial\Omega$,
\item $R[g] \geq 0$ in $\Omega$, 
\item $H[g] \geq H[g_{\mathbb{R}^n}]$ on $\partial\Omega$, and
\item $\|g - g_{\mathbb{R}^n}\|_{C^2(\overline{\Omega})} \leq \epsilon$,
\end{itemize}
there is a diffeomorphism $\varphi: {\Omega}\to {\Omega}$ such that $\varphi^*g = g_{\mathbb{R}^n}$ in $\Omega$
and $\varphi = id$ on $\partial\Omega$.
\end{theorem} 

\begin{proof} Again, in the light of Lemma \ref{slice}, we may assume that $g = g_{\mathbb{R}^n} + h$ and
$$
\delta h = 0 \text{ in $\Omega$ and } h|_{T\partial\Omega} = 0 \text{ on $\partial\Omega$}.
$$
Then, using the smoothness of the boundary $\partial\Omega$ and the fact that $h|_{T\partial\Omega} = 0$ on $\partial\Omega$, one
derives that
$$
\int_\Omega |\nabla h|^2 d\textup{vol}[\bar g] \geq \mu(\Omega) \int_\Omega |h|^2 d\textup{vol}[\bar g]
$$
for some positive number $\mu(\Omega)$, based on the argument similar to the one in the proof of Lemma \ref{Euclidean-eigen}. 
Therefore one can show that $h$ has to vanish in $\Omega$, whenever $\|h\|_{C^2(\overline{\Omega})}$ is sufficiently small. 
Thus the proof is complete.
\end{proof}

It is easily seen that, for example, \eqref{weak-convex} holds on convex domains including round balls in $\mathbb{R}^n$. 

\subsection{Hyperbolic spaces} The natural way to describe the hyperbolic space $\mathbb{H}^n$ is to identify it as the hyperboloid 
$$
\mathbb{H}^n = \{(t, x)\in \mathbb{R}^{n+1}: -t^2 + |x|^2 = -1 \text{ and } t > 0\}
$$
in the Minkowski space-time $(\mathbb{R}^{n+1}, \ -(dt)^2 + |dx|^2)$. In this coordinate 
$$
g_{\mathbb{H}^n} = \frac {(d|x|)^2}{1+|x|^2} + |x|^2 g_{\mathbb{S}^{n-1}}.
$$
With the lapse function $f = t = \sqrt{1 +|x|^2}$, the hyperbolic space $(\mathbb{H}^n, \ g_{\mathbb{H}^n})$ is a vacuum static space of negative cosmological constant
(cf. \cite{Qing}).
We will consider the geodesic balls $B^{\mathbb{H}}_r$ center from the vertex $(1, 0)$. We again calculate from  \eqref{I_Omega} and \eqref{bdy-term} that
$$
I_{B^{\mathbb{H}}_r}  =  \frac 14 \int_{B^{\mathbb{H}}_r} \left( |\nabla h|^2 + |d(\textup{Tr} h)|^2 - 2|h|^2 - 2 |\textup{Tr}h|^2 \right) t d\textup{vol}[g_{\mathbb{H}^n}]
$$
and
$$
\aligned
B_{B^{\mathbb{H}}_r}  & \geq  \int_{\partial B^{\mathbb{H}}_r} \left((2 - h_{nn})( H[g] - H[g_{\mathbb{H}^n}] )
+ (\frac 14 h_{nn}^2 + \frac n{2(n-1)}\sum_{i=1}^{n-1}h_{in}^2)H[g_{\mathbb{H}^n}] \right) \cosh rd\sigma[g_{\mathbb{H}^n}] \\
& + \int_{\partial B^{\mathbb{H}}_r} (\frac 34 h_{nn}^2 + \frac 14|h|^2)\sinh rd\sigma[g_{\mathbb{H}^n}]  
+ O \left( \int_{\partial B^{\mathbb{H}}_r} |h|^2 (|\nabla h| +  |h|) d\sigma[g_{\mathbb{H}^n}] \right),
\endaligned
$$
where $t = \cosh r$ and $\partial_\nu t = \sinh r$ are positive. 
\begin{theorem}\label{Hyperbolic_rigidity}
For $n\geq 3$, let $B_r^{\mathbb{H}}$ be the geodesic ball centered at the vertex $(1, 0)$ with radius $r > 0$ on the hyperboloid. 
There exists an $\varepsilon_0 > 0$, such that,  for any metric $g$ on $B_r^{\mathbb{H}}$ satisfying
\begin{itemize}
  \item $g = g_{\mathbb{H}^n}$ on $\partial B_r^{\mathbb{H}}$;
  \item$R[g] \geq -n(n-1)$ in $B_r^{\mathbb{H}}$; 
  \item $H[g] \geq H[g_{\mathbb{H}^n}]$ on $\partial B_r^{\mathbb{H}}$;
  \item $|| g - g_{\mathbb{H}^n}||_{C^2 (\overline{B_r^{\mathbb{H}}})} < \varepsilon_0$, 
\end{itemize}
there is a diffeomorphism $\varphi: {B_r^{\mathbb{H}}} \to {B_r^{\mathbb{H}}}$ such that $\varphi^*g = g_{\mathbb{H}^n}$ in $B_r^{\mathbb{H}}$
and $\varphi = id$ on $\partial B_r^{\mathbb{H}}$.
\end{theorem}

\begin{proof} Again, in the light of Lemma \ref{slice}, we may assume that $g = g_{\mathbb{H}^n} + h$ and
$$
\delta h = 0 \text{ in $\Omega$ and } h|_{T\partial\Omega} = 0 \text{ on $\partial\Omega$}.
$$
From the assumptions we have 
\begin{equation}\label{first-equation}
\aligned
I _{B^{\mathbb{H}}_r} + B_{B^{\mathbb{H}}_r} &  \geq \frac 14 \int_{B^{\mathbb{H}}_r} \left( |\nabla h|^2 
+ |d(\textup{Tr} h)|^2 - 2|h|^2 -2 |\textup{Tr} h|^2 \right) t d\textup{vol}[g_{\mathbb{H}^n}] + \frac 14\int_{\partial B^{\mathbb{H}}_r} |h|^2\cosh rd\sigma[g_{\mathbb{H}^n}] \\
& +  \int_{\partial B^{\mathbb{H}}_r} (\frac 34 h_{nn}^2 + \frac 14|h|^2)\partial_\nu td\sigma[g_{\mathbb{H}^n}]  
 -  C \int_{\partial B^{\mathbb{H}}_r} |h|^2 (|\nabla h| +  |h|) d\sigma[g_{\mathbb{H}^n}]
 \endaligned
\end{equation}
for some constant $C>0$. Similar to the idea used in \cite{C-M-T}, we  want to use the positive boundary terms to help to cancel the negative interior terms. For that,   
we perform integral by parts and estimate
$$
\aligned
 \int_{\partial B^{\mathbb{H}}_r} & (|\textup{Tr}h|^2 + |h|^2)\partial_\nu td\textup{vol}[g_{\mathbb{H}^n}] 
 =  \int_{B^{\mathbb{H}}_r} \text{div}((|\textup{Tr}h|^2 + |h|^2)\nabla t)d\textup{vol}[g_{\mathbb{H}^n}] \\
 & =  \int_{B^{\mathbb{H}}_r} (|\textup{Tr}h|^2 + |h|^2)\Delta t d\textup{vol}[g_{\mathbb{H}^n}] + 2
 \int_{B^{\mathbb{H}}_r} ((\textup{Tr}h)\nabla (\textup{Tr}h) + h\nabla h)\cdot\nabla t d\textup{vol}[g_{\mathbb{H}^n}]\\
 &\geq  n \int_{B^{\mathbb{H}}_r} (|\textup{Tr}h|^2 + |h|^2)td\textup{vol}[g_{\mathbb{H}^n}] - 2 \int_{B^{\mathbb{H}}_r} 
 (|\textup{Tr}h||\nabla \textup{Tr}h| + |h||\nabla h|)|\nabla t |d\textup{vol}[g_{\mathbb{H}^n}]\\
 &\geq n \int_{B^{\mathbb{H}}_r} (|\textup{Tr}h|^2 + |h|^2)td\textup{vol}[g_{\mathbb{H}^n}] -  \int_{B^{\mathbb{H}}_r} 
 ( a(|\textup{Tr}h|^2+|h|^2) + \frac 1a(|\nabla \textup{Tr}h|^2 +  |\nabla h|^2)) t d\textup{vol}[g_{\mathbb{H}^n}].
 \endaligned
$$
Here we use the fact that $\Delta t = n t$ and $|\nabla t|< t$. Going back to \eqref{first-equation} we get, for the choices $b= \frac 34$ and $a = \frac {11}6$, 
\begin{align*}
I _{B^{\mathbb{H}}_r} + B_{B^{\mathbb{H}}_r} &  \geq  \int_{B^{\mathbb{H}}_r} \left((\frac 14 - \frac {(1+b)}{4a}) (|\nabla h|^2 
+ |d(\textup{Tr} h)|^2) + (\frac {(1+b)}4(n-a) - \frac 12) (|h|^2 + |\textup{Tr} h|^2) \right) t d\textup{vol}[g_{\mathbb{H}^n}] \\
& +\frac {1-b}4 \int_{\partial B^{\mathbb{H}}_r} |h|^2\cosh rd\sigma[g_{\mathbb{H}^n}]+\frac {2-b}4 \int_{\partial B^{\mathbb{H}}_r} h_{nn}^2 \sinh rd\sigma[g_{\mathbb{H}^n}]
-  C \int_{\partial B^{\mathbb{H}}_r} |h|^2 (|\nabla h| +  |h|) d\sigma[g_{\mathbb{H}^n}]\\
 & =   \int_{B^{\mathbb{H}}_r} \left(\ \frac 1{88} (|\nabla h|^2 
+ |d(\textup{Tr} h)|^2) + \frac 7{16}(n-3 +\frac 1{42}) (|h|^2 + |\textup{Tr} h|^2) \right) t d\textup{vol}[g_{\mathbb{H}^n}] \\
& +\frac 1{16} \int_{\partial B^{\mathbb{H}}_r} |h|^2\cosh rd\sigma[g_{\mathbb{H}^n}] +\frac {5}{16} \int_{\partial B^{\mathbb{H}}_r} h_{nn}^2 \sinh rd\sigma[g_{\mathbb{H}^n}]
-  C \int_{\partial B^{\mathbb{H}}_r} |h|^2 (|\nabla h| +  |h|) d\sigma[g_{\mathbb{H}^n}].
\end{align*}
Now one may conclude that $h=0$, when $||h||_{C^2 (\overline{B_r^{\mathbb{H}}})}$ is sufficiently small and $n\geq 3$.
\end{proof}

\subsection{Hemispheres} The upper hemisphere $\mathbb{S}^n_+$ with the standard round metric is a vacuum static space of positive 
cosmological constant, where the lapse function $f = x_{n+1}$ is the high function when the sphere is the unit round sphere $\mathbb{S}^n$ centered at the origin in 
the Euclidean space $\mathbb{R}^{n+1}$. One considers the geodesic balls $B^{\mathbb{S}}_r$ centered at the north pole on the hemisphere. One then calculates
from  \eqref{I_Omega} and \eqref{bdy-term} that
$$
I_{B^{\mathbb{S}}_r}  =  \frac 14 \int_{B^{\mathbb{S}}_r} \left( |\nabla h|^2 + |d(\textup{Tr} h)|^2 + 2 |h|^2 + 2|\textup{Tr} h|^2 \right) x_{n+1} d\text{vol}[g_{\mathbb{S}^n}] 
$$
and
$$
\aligned
B_{B^{\mathbb{S}}_r} = &  \int_{\partial B^{\mathbb{S}}_r} \left( (2 - h_{nn}) \left( H[g] - H[g_{\mathbb{R}^n}] \right) + 
\left( \frac{1}{4}h^2_{nn} + \frac{n}{2(n-1)}\sum_{i=1}^{n-1} h^2_{in}\right) H[g_{\mathbb{R}^n}]\right) \cos r d\sigma[g_{\mathbb{S}^n}] \\
& - \int_{\partial B^{\mathbb{S}}_r} \left( h^2_{nn} + \frac{1}{2} \sum_{i=1}^{n-1}h^2_{in} \right) \sin r d\sigma[g_{\mathbb{S}^n}] 
+ O \left( \int_{\partial B^{\mathbb{S}}_r} |h|^2 (|\nabla h| +  |h|) d\sigma[g_{\mathbb{S}^n}] \right)\\
& \geq  \int_{\partial B^{\mathbb{S}}_r} \left(\left( \frac{n-1}{4} \frac{\cos^2 r}{\sin r} -\sin r \right)  h^2_{nn} 
+ \frac{1}{2}\left( n \frac{\cos^2 r}{\sin r} -\sin r \right)\sum_{i=1}^{n-1} h^2_{in} \right)d\sigma[g_{\mathbb{S}^n}] \\
&\quad\quad - C\int_{\partial B^{\mathbb{S}}_r} |h|^2 (|\nabla h| +  |h|) d\sigma[g_{\mathbb{S}^n}],
\endaligned
$$
where $f = f(r) = \cos r$ and $H[g_{\mathbb{S}^n}] = (n-1)\cot r$ for $\partial B^{\mathbb{S}}_r$ in the hemisphere.

\begin{theorem}\label{Thm:B_M} (\cite{B-M}) 
Consider the geodesic ball $B^{\mathbb{S}}_r$ with $\cos r \geq \frac{2}{\sqrt{n+3}}$. Let $g$ be a Riemannian metric on $B^{\mathbb{S}}_r$ with the following properties:
\begin{itemize}
\item $R[g] \geq n(n-1)$ in $B^{\mathbb{S}}_r$;
\item $H[g] \geq (n-1)\cot r$ on $\partial B^{\mathbb{S}}_r$;
\item $g$ and $g_{\mathbb{S}^n}$ induced the same metric on $\partial B^{\mathbb{S}}_r$,
\end{itemize}
If $g- g_{\mathbb{S}^n}$ is sufficiently small in the $C^2$-norm, then $\varphi^* (g) = g_{\mathbb{S}^n}$ for some diffeomorphism 
$\varphi: B^{\mathbb{S}}_r \rightarrow B^{\mathbb{S}}_r$ with $\varphi|_{\partial B^{\mathbb{S}}_r } = id$.
\end{theorem}

\begin{remark}
In \cite{C-M-T},  the interior integral 
$$
I_{B^{\mathbb{S}}_r}  =  \frac 14 \int_{B^{\mathbb{S}}_r} \left( |\nabla h|^2 + |d(\textup{Tr} h)|^2 + 2 |h|^2 + 2|\textup{Tr} h|^2 \right) x_{n+1} d\text{vol}[g_{\mathbb{S}^n}] 
$$
is used cleverly to improve the size of the geodesic ball $B^{\mathbb{S}}_r$ that is bigger than $\cos r \geq \frac{2}{\sqrt{n+3}}$ of Theorem \ref{Thm:B_M} of \cite{B-M}.
\end{remark}


\section{Conformal Rigidity of Static Space}\label{Conformal_Rigidity_Section}

In this section we consider the scalar curvature rigidity among conformal deformations. This is inspired by the work in \cite{H-W1}, where the scalar curvature
rigidity among conformal deformations of the hemispheres is established.  For $n \geq 2$, let $(M^n,\bar{g}, f)$ be a static space with positive scalar curvature 
$R[\bar{g}] > 0$. We denote $$\Omega^+ = \{ x \in M : f(x) > 0\}.$$

\begin{theorem}
Let $(M^n,\bar{g}, f)$ be a complete $n$-dimensional static space with $R_{\bar{g}} > 0$ ($n \geq 2$). Assume $\Omega^+$ is a pre-compact subset in $M$. 
Then, if a metric $g \in [\bar{g}]$ on $M$ satisfies that
\begin{itemize}
\item $R[g] \geq R[\bar{g}]$ in $\Omega^+$,
\item $g$ and $\bar{g}$ induced the same metric on $\partial \Omega^+$, and
\item $H[g] = H[\bar{g}]$ on $\partial \Omega^+$,
\end{itemize}
then $g =\bar{g}$. 
\end{theorem}

\begin{proof}  Since $g\in [\bar g]$ we may write as usual
$$
g = \left\{ \aligned  e^{2u}\bar g  & \quad \text{ when $n=2$}\\ 
u^\frac 4{n-2}\bar g & \quad \text{ when $n\geq 3$}.
\endaligned\right.
$$
Hence
$$
R[g] = \left\{\aligned e^{-2u} \left( R[\bar{g}] - 2 \Delta u \right)  & \text{ when $n=2$}\\ 
 u^{-\frac{n+2}{n-2}} \left( R[\bar{g}] u - \frac {4(n-1)}{n-2}\Delta u \right) & \text{ when $n\geq 3$} 
 \endaligned\right.
\text{ and } 
 H[g] = \left\{ \aligned H[\bar g] + 2\partial_\nu u   & \text{ when $n=2$}\\ 
 H[\bar g] + \frac {2(n-1)}{n-2}\partial_\nu u & \text{ when $n\geq 3$}.
\endaligned\right.
 $$
If we let 
\begin{align*}
\Lambda (x) =
\begin{cases}
\frac{ R[\bar{g}] (e^{2u(x)} - 1)}{2 u(x)} = R[\bar g] e^{2\xi} & \quad \text{ when $n = 2$} \\
\frac{\frac {n-2}{4(n-1)} R[\bar{g}] u(x) \left(u(x)^{\frac{4}{n-2}} - 1\right)}{u(x)-1} = \frac {R[\bar g]}{n-1} \xi^\frac 4{n-2} \frac u \xi 
& \quad \text{ when $n\geq 3$},
\end{cases}
\end{align*}
where $\xi$ is between $0$ and $u(x)$ when $n=2$; $\xi$ is between $1$ and $u(x)$ when $n\geq 3$, and
\begin{align*}
v(x) = \begin{cases} u(x) & \quad \text{ when $n = 2$} \\
u(x) - 1& \quad \text{ when $n\geq 3$},
\end{cases}
\end{align*}
then we may rewrite the assumptions $R[g]\geq R[\bar g]$ and $H[g]\geq H[\bar g]$ as follows:
\begin{equation}\label{scalar-cur}
\left\{\aligned - \Delta v - \Lambda(x) v & \geq 0 \text{ in $\Omega^+$}\\
v & = 0 \text{ on $\partial\Omega^+$}\\
\partial_\nu v & = 0 \text{ on $\partial\Omega^+$}.
\endaligned\right.
\end{equation}
On the other hand, if denote $\Lambda = \frac {R[\bar g]}{n-1}$, we deduce from the static equation that 
\begin{equation}\label{lapse-equ}
- \Delta f - \Lambda f = 0 \text{ and } f > 0 \text{ in $\Omega^+$}.
\end{equation}
In the following we want to first use the positive lapse function $f$ in $\Omega^+$ to show $v\geq 0$. To do so we consider the quotient 
$\varphi(x) = \frac{v(x)}{f(x)}$ in $\Omega^+$ and calculate
\begin{equation}\label{quotient-equ}
\aligned
- \Delta \varphi &= \frac{1}{f} \left( - \Delta v + \varphi \Delta f  + 2 \nabla \varphi\cdot \nabla f \right)\\
& \geq - \varphi \left(\Lambda - \Lambda(x)\right) + 2 \nabla\varphi\cdot\frac {\nabla f}f.
\endaligned
\end{equation}
Assume otherwise that there exists $\bar{x} \in \Omega^+$ such that $\varphi(\bar{x}) < 0$. In order to apply the maximum principle we would like to use L$'$hospital$'$s 
rule to see that $\varphi = 0 $ on $\partial\Omega^+$. Here we need to use that fact that $\nabla f\neq 0$ at $\partial\Omega^+ = \{x\in M^n: f(x) = 0\}$ according
\cite[Theorem 1]{F-M}. Therefore we may assume that $\bar x$ be a minimum point for $\varphi$. Notice that $\Lambda (\bar{x}) < \Lambda$ when $\varphi(\bar{x}) <0$.
Thus, from \eqref{quotient-equ}, we arrive at
$$
0 \geq -\Delta \varphi (\bar{x}) \geq - \varphi (\bar{x}) \left(\Lambda - \Lambda(\bar{x})\right) > 0,
$$
which is a contradiction. Therefore we have shown that $v(x)\geq 0$ in $\Omega^+$. 
\\

Finally, applying the Hopf maximum principle, for instance, \cite[Theorem 7.3.3]{C-L}, to \eqref{scalar-cur},  we conclude that
$v \equiv 0$ in $\Omega^+$, that is,  $ g \equiv \bar g$ in $\Omega^+$. So the proof is compete.
\end{proof}

Next we want to show that the domain $\Omega^+$ is the biggest of which the scalar curvature rigidity holds. This generalizes the work in \cite{H-W1}. 
Our construction is different from \cite{H-W1} and is based on the idea in \cite{B-M-N},  particularly \cite[Theorem 5 and Lemma 21]{B-M-N}.

\begin{theorem} Let $(M^n,\bar{g}, f)$ be a complete $n$-dimensional static space with $R_{\bar{g}} > 0$ ($n \geq 2$). Assume the level set $\Omega^+$ is 
a pre-compact subset in $M$.  For any open domain $\Omega$ in $M^n$ that contains $\overline{\Omega^+}$ , there is a smooth 
metric $g \in [\bar{g}]$ such that 
\begin{itemize}
\item $R[g] > R[\bar{g}]$ at some point in $\Omega$ and
\item $\textup{supp}(g - \bar{g}) \subset \Omega$
\end{itemize}
\end{theorem}

\begin{proof} First, by \cite[Theorem 1]{F-M} , we know $\nabla f\neq 0$ on $\partial\Omega^+$. Hence we may assume 
the level set $\Omega^{-\epsilon} = \{x\in M^n: f(x) > -\epsilon\}\subset \Omega$, at least for sufficiently small $\epsilon$, is pre-compact in $M^n$. 
In $\Omega^{-\epsilon}$, we consider $u_t = 1 - t(f +\epsilon) >0$ on $M$ for $t\in [0, \delta]$ and
sufficiently small $\delta >0$, and the family of conformal deformations 
$$g_t = u_t^{\frac{4}{n-2}} \bar{g} \text{ when $n\geq 3$ and } g_t = u^2_t \bar g \text{ when $n=2$, for } t \in[0, \delta).$$ 
Then, when $n\geq 3$,  we calculate the expansion of scalar curvature, in $\Omega^{-\epsilon}$,
\begin{align*}
R[g_t] &= u_t^{-\frac{n+2}{n-2}} (R_{\bar{g}} u_t - \frac {4(n-2)}{n-1}\Delta u_t)\\
&= R_{\bar{g}} - \frac {4(n-1)}{n-2} \left( \Delta u_t + \frac{R[\bar{g}]}{n-1} u_t \right) t + O(t^2).\\ 
& = R_{\bar{g}} + \frac{4(n-2)}{(n-1)^2} \epsilon R[\bar{g}] t+ O (t^2)
\end{align*}
and  the expansion of the mean curvature of $\partial\Omega^{-\epsilon}$, 
\begin{align*}
H[g_t] = H_{\bar{g}} + \frac{2(n-1)t}{n-2} |\nabla f |.
\end{align*}
Similarly, when $n=2$, we have $R[g_t] = R[\bar g] + 2\epsilon t R[\bar g] + O(t^2)$ and $ H[g_t] = H[\bar g] +  t|\nabla f|$.
Hence we constructed a conformal metric $g_t$ in $\Omega^{-\epsilon}$ for some sufficiently small $t$ such that
\begin{itemize}
\item  $R[g_t]  > R[\bar{g}]$ in $\Omega^{-\epsilon}$;
\item $g_t = \bar g$ on $\partial\Omega^{-\epsilon}$;
\item $H[g_t] > H[\bar g]$ on $\partial\Omega^{-\epsilon}$.
\end{itemize}
Next we want to use \cite[Theorem 5]{B-M-N} to construct a smooth metric $g\in[\bar g]$ so that 
\begin{itemize}
\item  $R[g]  \geq R[\bar{g}]$ in $\Omega$, $R[g] > R[\bar g]$ at some point in $\Omega$, and
\item $\textup{supp}(g - \bar g)\subset \Omega$.
\end{itemize}
To do so we need a smooth metric in the neighborhood of $\partial\Omega^{-\epsilon}$ that smoothly extends to outside $\Omega^{-\epsilon}$. 
Based the same idea used in \cite[Lemma 21]{B-M-N}, we consider
$$
\tilde g = w^\frac 4{n-2} \bar g \text{ when $n\geq 3$ and }  \tilde g = w^2 \bar g \text{ when $n=2$}, 
$$
for $w = 1- e^{-\frac 1{f+\epsilon}}$ in $\Omega^{-\epsilon}_0 = \{x\in M^n: -\epsilon < f(x) < 0\}$. Similar calculations as in the proof of \cite[Lemma 21]{B-M-N} show that
\begin{itemize}
\item $R[\tilde g]> R[\bar g]$ in $\Omega^{-\epsilon}_0$ and
\item $\textup{supp}(\tilde g - \bar g)\subset\Omega$.
\end{itemize} 
Particularly, $H[\tilde g] = H[\bar g]$ at $\partial\Omega^{-\epsilon}$. Now, applying \cite[Theorem 5]{B-M-N} to glue $g_t$ and $\tilde g$ on $\Omega^{-\epsilon}$,
we get a smooth metric $g$ that finishes the proof,  so long as one realizes that $g$ stays in $[\bar g]$ if both $\tilde g$ and $g_t$ are conformal to $\bar g$ due to the construction
of the gluing in \cite{B-M-N}.
\end{proof}


\bibliographystyle{amsplain}

\end{document}